\documentclass[12pt]{article}
\usepackage{amssymb,amsthm,amsmath,amsfonts,latexsym,tikz,tikz-cd,hyperref}
\usepackage[hmargin=1in,vmargin=1in]{geometry}

\newcommand{\ben}{\begin{enumerate}}
\newcommand{\een}{\end{enumerate}}
\newcommand{\ble}{\begin{lem}}
\newcommand{\ele}{\end{lem}}
\newcommand{\bth}{\begin{thm}}
\renewcommand{\eth}{\end{thm}}
\newcommand{\bpr}{\begin{prop}}
\newcommand{\epr}{\end{prop}}
\newcommand{\bco}{\begin{cor}}
\newcommand{\eco}{\end{cor}}
\newcommand{\bcon}{\begin{conj}}
\newcommand{\econ}{\end{conj}}
\newcommand{\bde}{\begin{defn}}
\newcommand{\ede}{\end{defn}}
\newcommand{\bex}{\begin{exa}}
\newcommand{\eex}{\end{exa}}
\newcommand{\barr}{\begin{array}}
\newcommand{\earr}{\end{array}}
\newcommand{\btab}{\begin{tabular}}
\newcommand{\etab}{\end{tabular}}
\newcommand{\beq}{\begin{equation}}
\newcommand{\eeq}{\end{equation}}
\newcommand{\bea}{\begin{eqnarray*}}
\newcommand{\eea}{\end{eqnarray*}}
\newcommand{\bal}{\begin{align*}}
\newcommand{\bce}{\begin{center}}
\newcommand{\ece}{\end{center}}
\newcommand{\bpi}{\begin{picture}}
\newcommand{\epi}{\end{picture}}
\newcommand{\bpp}{\begin{picture}}
\newcommand{\epp}{\end{picture}}
\newcommand{\bfi}{\begin{figure} \begin{center}}
\newcommand{\efi}{\end{center} \end{figure}}
\newcommand{\bprf}{\begin{proof}}
\newcommand{\eprf}{\end{proof}\medskip}
\newcommand{\capt}{\caption}
\newcommand{\bsl}{\begin{slide}{}}
\newcommand{\esl}{\end{slide}}
\newcommand{\bfr}{\begin{frame}}
\newcommand{\efr}{\end{frame}}

\newcommand{\hqed}{\hfill \qed}

\newcommand{\eqed}[1]{$\textcolor{white}{\qed}\hfill{\dil#1}\hfill\qed$}

\newcommand{\ol}{\overline}

\newcommand{\hso}[1]{\hspace{-1pt}}
\newcommand{\vs}[1]{\vspace{#1}}

\newcommand{\emp}{\emptyset}

\newcommand{\sbe}{\subseteq}

\newcommand{\iso}{\cong}

\newcommand{\lt}{\lhd}
\newcommand{\gt}{\rhd}
\newcommand{\lte}{\unlhd}
\newcommand{\gte}{\unrhd}

\newcommand{\case}[4]{\left\{\barr{ll}#1&\mbox{#2}\\#3&\mbox{#4}\earr\right.}

\def\<{\langle}
\def\>{\rangle}

\newcommand{\ra}{\rightarrow}

\newcommand{\al}{\alpha}

\newcommand{\de}{\delta}

\newcommand{\ka}{\kappa}

\newcommand{\De}{\Delta}

\newcommand{\bbN}{{\mathbb N}}

\newcommand{\cA}{{\cal A}}

\newcommand{\cG}{{\cal G}}

\newcommand{\cI}{{\cal I}}

\newcommand{\cL}{{\cal L}}

\newcommand{\cM}{{\cal M}}

\newcommand{\cO}{{\cal O}}
\newcommand{\cP}{{\cal P}}

\newcommand{\cS}{{\cal S}}

\newcommand{\ab}{\ol{a}}

\newcommand{\Ib}{\ol{I}}

\newcommand{\Jb}{\ol{J}}

\DeclareMathOperator{\lcm}{lcm}

\DeclareMathOperator{\st}{st}

\newcommand{\dil}{\displaystyle}

\newtheorem{thm}{Theorem}[section]
\newtheorem{prop}[thm]{Proposition}
\newtheorem{cor}[thm]{Corollary}
\newtheorem{lem}[thm]{Lemma}
\newtheorem{conj}[thm]{Conjecture}
\newtheorem{exa}[thm]{Example}
\newtheorem{question}[thm]{Question}

\theoremstyle{definition}
\newtheorem{definition}[thm]{Definition}

\newcommand{\uF}{\breve{F}}
\newcommand{\uS}{\breve{S}}
\newcommand{\cU}{{\mathcal U}}
\newcommand{\na}{\nabla}
\newcommand{\rhoh}{\hat{\rho}}

\newcommand{\chih}{\hat{\chi}}

\newcommand{\cOb}{\ol{\cO}}
\newcommand{\bba}{\ol{b}}

\usepackage[colorinlistoftodos,color=red!70]{todonotes}

\newcommand{\re}[1]{\filldraw (current) rectangle ++ (#1-.3,.7) ++(.3,-.7) coordinate (current);}
\newcommand{\ro}{\filldraw[fill=yellow] (current) rectangle ++(.7,.7)  ++(.3,-.7)
coordinate (current);}
\newcommand{\twobyone}{\filldraw[fill=red] (current) ++(0,.7) rectangle ++(.7,-1.7) ++(.3,1) coordinate (current);}
\newcommand{\blank}{\draw (current) ++(1,0) coordinate (current);}
\newcommand{\releft}[1]{ \filldraw (current) -- ++(-.3,0.23) -- ++(.3,.24) -- ++(-.3,.23) -- ++(#1,0) -- ++(0,-.7) ++(.3,0) coordinate (current);}
\newcommand{\reright}[1]{ \filldraw (current) -- ++(0,.7)-- ++(#1-.3,0) -- ++(.3,-.23) -- ++(-.3,-.24) -- ++(.3,-.23);}
\newcommand{\firstrow}{\coordinate (current) at (.15,-.85);}
\newcommand{\secondrow}{\coordinate (current) at (.15,-1.85);}
\newcommand{\thirdrow}{\coordinate (current) at (.15,-2.85);}
\newcommand{\fourthrow}{\coordinate (current) at (.15,-3.85);}
\newcommand{\fifthrow}{\coordinate (current) at (.15,-4.85);}

\begin{document}
\pagestyle{plain}

\title{Rowmotion on fences
}
\author{Sergi Elizalde\\[-5pt]
\small{Department of Mathematics, Dartmouth College,}\\[-5pt]
\small{Hanover, NH 03755, 
\texttt{sergi.elizalde@dartmouth.edu}}\\
Matthew Plante\\[-5pt]
\small{Department of Mathematics, University of Connecticut,}\\[-5pt]
\small{Storrs, CT 06269, \texttt{matthew.plante@uconn.edu}}\\
Tom Roby\\[-5pt]
\small{Department of Mathematics, University of Connecticut,}\\[-5pt]
\small{Storrs, CT 06269, \texttt{tom.roby@uconn.edu}}\\
Bruce E. Sagan\\[-5pt]
\small Department of Mathematics, Michigan State University,\\[-5pt]
\small East Lansing, MI 48824, USA, {\tt sagan@math.msu.edu}
}

\date{
	\begin{flushleft}
	                                       \\[5pt]
	\end{flushleft}}

\maketitle

\begin{abstract}
A fence is a poset with elements $F=\{x_1,x_2,\ldots,x_n\}$ and covers
$$
  x_1\lt x_2 \lt \ldots \lt x_a \gt x_{a+1} \gt \ldots \gt x_b \lt x_{b+1} \lt \cdots  
$$
where $a,b,\ldots$ are positive integers.  We investigate rowmotion on antichains and ideals of $F$.  In particular, we show that orbits of antichains can be visualized using tilings.  This permits us to prove various homomesy results for the number of elements of an antichain or ideal in an orbit.  Rowmotion on fences also exhibits a new phenomenon, which we call homometry, where the value of a statistic is constant on orbits of the same size.
Along the way, we prove a general homomesy result for all self-dual posets. We end with some conjectures and avenues for future research.
\end{abstract}


\section{Introduction}

The purpose of this work is to initiate the study of the dynamical algebraic combinatorics of fence posets.  A {\em fence} is a poset with elements $F=\{x_1,x_2,\ldots,x_n\}$, partial order $\lte$, and covers
\begin{equation}
\label{covers}
  x_1\lt x_2 \lt \ldots \lt x_a \gt x_{a+1} \gt \ldots \gt x_b \lt x_{b+1} \lt \cdots  
\end{equation}
where $a,b,\ldots$ are positive integers.  The maximal chains of $F$ are called {\em segments}.   A fence with seven elements and three segments is shown in Figure~\ref{F(3,3,2)}.  
Throughout we will often use $n$ for the cardinality of $F$, which we denote by $\#F$. 
We also let $[n]=\{1,2,\ldots,n\}$.

There are a number of different conventions for indicating the size of the segments of a fence $F$ in the literature depending on the application being considered.  Our results will be simplest if described in terms of unshared elements. 
Call $x\in F$ {\em shared} if it is the intersection of two segments; otherwise $x$ is {\em unshared}.  In Figure~\ref{F(3,3,2)}, the elements $x_3$ and $x_6$ are shared and all other elements unshared.  Let $\al=(\al_1,\al_2,\ldots,\al_s)$ be a composition (a sequence of positive integers called {\em parts}) with $\al_1,\al_s\ge2$.  The corresponding fence is 
$$
F=\uF(\al)=\uF(\al_1,\al_2,\ldots,\al_s)
$$
where
$$
\al_i = 1+(\text{number of unshared elements on the $i$th segment})
$$
for $i\in[s]$.  The fence in Figure~\ref{F(3,3,2)} is $F=\uF(3,3,2)$.
Note that for any fence, 
$$
\#F=\al_1+\al_2+\dots+\al_s-1.
$$

\bfi
\begin{tikzpicture}
\fill(0,1) circle(.1);
\fill(1,2) circle(.1);
\fill(2,3) circle(.1);
\fill(3,2) circle(.1);
\fill(4,1) circle(.1);
\fill(5,0) circle(.1);
\fill(6,1) circle(.1);
\draw (0,1)--(2,3)--(5,0)--(6,1);
\draw (0,.5) node{$x_1$};
\draw (1,1.5) node{$x_2$};
\draw (2,2.5) node{$x_3$};
\draw (3,1.5) node{$x_4$};
\draw (4,.5) node{$x_5$};
\draw (5,-.5) node{$x_6$};
\draw (6,.5) node{$x_7$};
\end{tikzpicture}
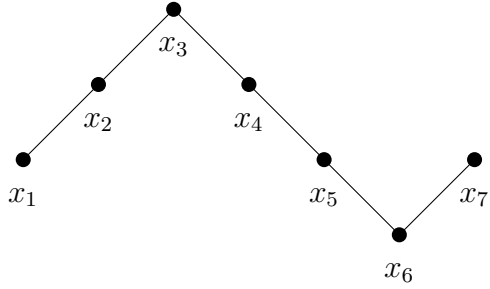
\capt{The fence $F=\uF(3,3,2)$ \label{F(3,3,2)}}
\efi

Fences have been of recent interest because of their connections with cluster algebras, $q$-analogues, and unimodality.  
Let $P$ be a poset with partial order $\lte$. Recall that $I\sbe P$ is a {\em (lower order) ideal} of $P$ if $x\in I$ and $y\lte x$ implies $y\in I$. {\em Upper order ideals}, $U$, are defined by reversing the inequality. 
When merely writing ``ideal," we always mean a lower order ideal.
Let $\cI(P)$ and $\cU(P)$ be the set of (lower order) ideals and upper order ideals of $P$, respectively. 
The set of ideals of a finite poset forms a distributive lattice under inclusion. 
The lattice $\cI(F)$ of ideals of a fence $F$ can be used to compute mutations in a cluster algebra on a surface with marked points~\cite{cla:epp,MSW:pca,pro:cfp,sch:cef,sch:caa,ST:caa,yur:cef,yur:cce}. 

Let $q$ be a variable and $r(F;q)$ be the rank generating function for $\cI(F)$.  Mourier-Genoud and Ovsienko~\cite{MGO:qdr} were able to define $q$-analogues of rational (and even real) numbers which are certain rational functions of $q$.  The numerators and denominators of these fractions are exactly the polynomials $r(F;q)$.  In addition, they conjectured the following result.  Progress on this question was made in~\cite{cla:epp,ES:prs,gan:loi,MSS:ruc,MZ:rpl} until a proof was provided by Oğuz and Ravichandran.
\bth[\cite{OR:rpf}]
The polynomial $r(F;q)$ is unimodal.
\eth

Our  focus is going to be on fences from the viewpoint of algebraic dynamical combinatorics and, in particular, on rowmotion.
A subset $A$ of a finite poset $(P,\lte)$ is an {\em antichain} if no two elements of $A$ are comparable.
Let $\cA(P)$ be the set of antichains of $P$.  
Ideals of both types and antichains are related by the maps
$\De:\cI(P)\ra\cA(P)$ where
$$
\De(I) = \{x\in I \mid \text{$x$ is a maximal element of $I$}\},
$$
and $\na:\cU(P)\ra\cA(P)$ where
$$
\na(U) = \{x\in U \mid \text{$x$ is a minimal element of $U$}\}.
$$
We also let $c:\cP\ra\cP$ be the complement operator $c(S)=P-S$.
{\em Rowmotion on antichains} is the group action on $\cA(P)$ generated by the map
$$
\rho = \na \circ c \circ \De^{-1}
$$
where we always compose functions from right to left.  We will also consider {\em rowmotion on ideals}, which is generated by
$$
\rhoh =  \De^{-1} \circ \na \circ c.
$$
There is clearly a bijection between the orbits of $\rho$ and those of $\rhoh$.
We will call the number of elements (that is, the number of antichains or, equivalently, the number of ideals) in an orbit either its {\em size} or its {\em length}.
Rowmotion and its generalizations have been investigated by many authors~\cite{DPS:rop,EP:pbt,EP:cpbh,GP:rs,GR:ipb2,GR:ipb1,jos:atr,MR:pub,str:rgt,SW:pr,TW:rsm,vor:hpt}.  See, in particular, the survey articles of Roby~\cite{rob:dac} and Striker~\cite{str:dac}.
We will let 
$$
\cI(\al)=\cI(\uF(\al))
$$ 
and similarly for $\cU$ and $\cA$.

In addition to describing the orbits of rowmotion on fences, we will also consider properties of various statistics.
If $S$ is a finite set, then a 
{\em statistic} on $S$ is a map
$\st:S\ra\bbN$ where $\bbN$ is the set of  nonnegative integers.
If $G$ is a group acting on $S$, then statistic $\st$ is 
{\em $d$-mesic}  if there is a constant $d$ such that every orbit $\cO$ of the action has average
$$
\frac{\st\cO}{\#\cO}=d,
$$
where $\st\cO=\sum_{x\in\cO} \st x$ and $\#\cO$ is the size of the orbit.
We say that $\st$ is {\em homomesic} if it is $c$-mesic for some $c$.  Homomesy is a well-studied property of rowmotion. 
The reader will find more information about this notion in the survey articles just cited.
Rowmotion on fences displays a new and interesting phenomenon.  We say that $\st$ is {\em homometric} if 
$\st\cO$ is constant over all orbits of the same cardinality.  Equivalently, for any two orbits $\cO_1$ and $\cO_2$ we have
$$
\#\cO_1 = \#\cO_2 \implies 
\st\cO_1 =\st\cO_2.
$$
Note that homomesy implies homometry, but not conversely.
In the sequel we will see many homometries which are not homomesies.

The rest of this paper is structured as follows.
In the next section we introduce our principal tool in this work, which is a representation of rowmotion orbits of antichains of fences in terms of certain tilings.
Section~\ref{has} is devoted to applying this model to prove various homomesy results for fences with any number of segments.  We also define, for any self-dual poset, a new orbit structure which is coarser than that of rowmotion, and prove a homomesy on ideals in this setting.
Section~\ref{ffs} is devoted to examining orbits and homometries for certain fences having at most five segments.  We end with a section containing conjectures and future directions.


\section{Tilings}
\label{til}

It turns out that the orbits of $\cA(\al)$ can be nicely visualized in terms of tilings.  This will be our principal tool in proving homo- and homometry results.

\bfi
\begin{tikzpicture}[scale=.5]
\draw[very thin] (0,0) grid (10,4);
\draw (-1,2) node{$\cdots$};
\draw (11,2) node{$\cdots$};
\draw(-3,3.5) node{$1$};
\draw(-3,2.5) node{$2$};
\draw(-3,1.5) node{$3$};
\draw(-3,.5) node{$4$};
\end{tikzpicture}
\capt{Part of the horizontal strip $H_4$}
\label{H_4}
\efi

Consider an infinite horizontal strip $H_s$ subdivided into unit squares with rows  numbered $1,2,\ldots,s$ from top to bottom, and infinitely many columns.  We display $H_4$ together with its row numbering in Figure~\ref{H_4}.  In the following definition, an {\em $a\times b$ tile} is a tile which covers $a$ rows and $b$ columns in $H_s$.

\begin{definition}
\label{TilDef}
Let $\al=(\al_1,\al_2,\ldots,\al_s)$ be a composition with $s$ parts.  An {\em $\al$-tiling} is a tiling of $H_s$  using yellow $1\times 1$ tiles, red $2\times 1$ tiles, and black $1\times(\al_i-1)$ tiles in row $i$, for $1\le i\le s$, such that the following hold for all rows.
\begin{enumerate}
    \item[(a)] If $\al_i>1$ and the red tiles are ignored, then the black and yellow  tiles alternate in row $i$.
    \item[(b)] There is a red tile in a column covering rows $i$ and $i+1$ if and only if either the previous column contains two yellow tiles in those two  rows when $i$ is even, or the next column contains two yellow tiles in those two  rows when $i$ is odd.
\end{enumerate}
\end{definition}

When $\al$ is clear from the context, we will use the term {\em tiling} to refer to an $\al$-tiling. 
We consider two tilings to be the same if one is a horizontal translate of the other.  It will follow from the proof of Lemma~\ref{tiling} that all $\al$-tilings are periodic and so can be viewed as lying on a cylinder.  When a tiling is displayed in a figure, we will draw a bounded rectangle and  assume that the two vertical edges  are identified, and indicate with a jagged edge where any black tile crosses this boundary.
See Figure~\ref{uF(4,3,4)} for the four possible $(4,3,4)$-tilings.  
Note that by considering tilings to be on a cylinder, notions such as the number of tiles of each color make sense.  For example, row $1$ of the top tiling in Figure~\ref{uF(4,3,4)} has four yellow tiles, four black tiles, and a red tile which also intersects row $2$.

In a tiling we will call a square of $H_s$ yellow, red, or black depending on whether the tile covering the square has the corresponding color.  The {\em head} of a red tile is the square it covers in the higher of the two rows.

\begin{figure}
\centering
\begin{tikzpicture}
\draw(-2,1.5) node{$\uF(4,3,4)=$};
\fill(-1,0) circle(.1);
\fill(0,1) circle(.1);
\fill(1,2) circle(.1);
\fill(2,3) circle(.1);
\fill(3,2) circle(.1);
\fill(4,1) circle(.1);
\fill(5,0) circle(.1);
\fill(6,1) circle(.1);
\fill(7,2) circle(.1);
\fill(8,3) circle(.1);
\draw (-1,0)--(2,3)--(5,0)--(8,3);
\draw (-1,-.5) node{$x_1$};
\draw (0,.5) node{$x_2$};
\draw (1,1.5) node{$x_3$};
\draw (2,2.5) node{$x_4$};
\draw (3,1.5) node{$x_5$};
\draw (4,.5) node{$x_6$};
\draw (5,-.5) node{$x_7$};
\draw (6,.5) node{$x_8$};
\draw (7,1.5) node{$x_9$};
\draw (8,2.5) node{$x_{10}$};
\end{tikzpicture}

\vs{20pt}

\begin{tikzpicture}[scale=0.5]
\draw[very thin] (0,0) grid (17,-3);
\thirdrow \releft{2} \ro \re{3} \ro \re{3} \ro \blank \re{3} \ro \reright{1}
\secondrow \blank \ro \re{2} \ro \re{2} \ro \re{2} \ro \twobyone \re{2} \ro \re{2} 
\firstrow \twobyone \ro \re{3} \ro \re{3} \ro \re{3} \ro \re{3};
\begin{scope}[shift={(0,-4)}]
\draw[very thin] (0,0) grid (17,-3);
\thirdrow \re{3} \ro \re{3} \ro \blank \re{3} \ro  \re{3} \ro
\secondrow \blank \ro \re{2} \ro \re{2} \ro \twobyone \re{2} \ro  \re{2} \ro \re{2}
\firstrow \twobyone \ro \re{3} \ro \re{3} \ro \re{3} \ro \re{3}
\end{scope}
\begin{scope}[shift={(0,-8)}]
\draw[very thin] (0,0) grid (17,-3);
\thirdrow \ro \re{3} \ro \blank \re{3} \ro \re{3} \ro  \re{3} 
\secondrow \blank \ro \re{2} \ro \twobyone \re{2} \ro \re{2} \ro  \re{2} \ro \re{2} 
\firstrow \twobyone \ro \re{3} \ro \re{3} \ro \re{3} \ro \re{3}
\end{scope}
\begin{scope}[shift={(19,-4)}]
\draw[very thin] (0,0) grid (5,-3);
\thirdrow \ro \blank \re{3}
\secondrow \ro \twobyone \re{2}
\firstrow \ro \re{3} \twobyone
\end{scope}
\end{tikzpicture}
\caption{The three antichain orbits of length $17$  and the orbit of length $5$  in $\uF(4,3,4)$.}
\label{uF(4,3,4)}
\end{figure}

Given $\al=(\al_1,\ldots,\al_s)$ we will now construct a bijection
$$
\phi:\{\cO \mid \text{$\cO$ an orbit of $\cA(\al)$}\}
\ra\{T \mid \text{$T$ an $\al$-tiling}\}
$$
as follows.  
Let the $i$th segment of $F(\al)$ be $S_i$.
Given $\cO$, we build $T=\phi(\cO)$ column-by-column.  Pick any $A\in\cO$ and any column $C$  of $H_s$ to correspond to $A$.  Color the square in row $i$ of $C$ yellow, red or black depending upon whether $S_i\cap A$ is empty, a shared element, or an unshared element, respectively.  
For example, the antichain $A=\{x_4,x_9\}$ in $\uF(4,3,4)$ corresponds to the first column of the uppermost tiling in Figure~\ref{uF(4,3,4)}.
Now color the column to the right of $C$ in $H_s$ in the same way using the antichain following $A$ in $\cO$, and similarly for the column to the left and the antichain preceding $A$.  Continue this process until all of $H_s$ is colored.  Clearly this is a periodic tiling and so can be wrapped onto a cylinder.  The tilings for the four antichain orbits in $\uF(4,3,4)$ are displayed in Figure~\ref{uF(4,3,4)}.  For example, the tiling with five columns corresponds to the orbit
$$
\{
\emp,\ 
\{x_1, x_7\},\ 
\{x_2, x_6, x_8\},\
\{x_3, x_5, x_9\},\
\{x_4, x_{10}\}
\}.
$$
\ble
\label{tiling}
For every $\al$, the map $\phi$ is a bijection.
\ele
\begin{proof}
We must first prove that $\phi$ is well defined in the sense that if $\cO$ is an orbit of antichains, then $T=\phi(\cO)$ is an $\al$-tiling. To start, we need to show that the colored squares can be partitioned into tiles of the appropriate sizes.  Every individual yellow square is clearly a $1\times 1$ tile.  And the red squares form $2\times 1$ tiles since any shared element of $F=F(\al)$ is in two adjacent segments.  The black tiles will take more work.

Suppose $x,y$ are unshared elements with $x$ covered by $y$.  We claim that $x\in A$ for some antichain $A$ if and only if $y\in\rho(A)$.  This will ensure  that in row $i$, any set of consecutive black squares will be of length at least $\al_i-1$, since that is the number of unshared elements in the $i$th segment $S_i$.  If $x\in A$, then $y$ is in the complement $U$ of the lower order ideal generated by $A$ since $y\gt x$.  And since $y$ is an unshared cover, $y$ is a minimal element of $U$.  This implies $y\in\rho(A)$.  The reverse implication is similar and left to the reader.

For the rest of the proof there are two cases depending on whether $i$ is even or odd.  We will just give details for the former, as the other case is very similar. From the previous paragraph, to show that sets of consecutive black squares have length exactly $\al_i-1$, it suffices to prove that, if $x\in A$ is the smallest  (respectively largest) unshared element on $S_i$, then the square before (respectively after) the one for $x$ is yellow.   Suppose that $x$ is the largest unshared element of $S_i$.  There are two subcases depending on whether $z\in A$ or $z\not\in A$, where $z$ is the largest unshared element of segment $S_{i-1}$.  In the first subcase, the maximal element $S_i\cap S_{i-1}$ is in $\rho(A)$, and so the square after the one for $x$ is covered by a red tile with its head in row $i-1$.  If $z\not\in A$, then $U$ contains an element of $S_{i-1}$ below its maximal element.  It follows that $\rho(A)\cap S_i=\emp$ and so the square after $x$ is a yellow tile.  In much the same way, one sees that if $x$ is the smallest unshared element of $S_i$ then the square before it is covered by a red tile with its head in row $i$ or a yellow tile depending on whether $w\in A$ or $w\not\in A$, where $w$ is the smallest unshared element of $S_{i+1}$.

We now verify condition (b) in Definition~\ref{TilDef}, postponing condition (a) until the next paragraph.  First suppose the column for $A$ contains a red tile covering rows $i$ and $i+1$.  Since $i$ is even, this corresponds to the minimal element $S_i\cap S_{i+1}$ being in $A$.  But this forces the intersections of $S_i$ and $S_{i+1}$ with $\rho^{-1}(A)$ to be empty.  It follows that the previous column contains yellow squares in those rows as desired.  All these implications are reversible, proving the converse statement.

Next we verify condition (a). It follows from the two previous paragraphs that any black tile is bounded on the left and on the right  by yellow tiles if one ignores the red ones.  Using similar arguments, one sees that between any two yellow tiles there must be a black tile as long as $\al_i\ge2$, so that black tiles can exist in row $i$.  So (a) follows and we have shown that the map $\phi$ is well defined. 

To show that $\phi$ is a bijection, we describe  its inverse.  Given an $\al$-tiling, $T$, we construct a family $\cO$ of subsets of $F$ as follows.
Each column $C$ of $T$ corresponds to an element $A\in\cO$.  
If the square in row $i$ of $C$ is yellow, then we 
let $A$ contain none of the elements of $S_i$.  If this square is a part of a red tile which covers row $i-1$ or $i+1$, then we let
$S_{i-1}\cap S_i$ or $S_i\cap S_{i+1}$ be in $A$, respectively.  Finally, if the square is the $j$th element of a black tile, then we put the $j$th smallest unshared element of $S_i$ in $A$.  Verifying that $A$ is an antichain and that the sequence of antichains forms an orbit does not involve any new ideas and so is left to the reader.  The fact that this is the inverse of $\phi$ is clear since the first part of the proof shows that this is true on the level of the individual columns of the tiling and antichains of the orbit.
\end{proof}

Using the tiling model it is easy to read off various statistics about orbits which will be useful in proving homomesy and homometry results.  For $x\in F$, consider the {\em indicator function on antichains} $\chi_x:\cA(F)\ra\{0,1\}$ defined by
$$
\chi_x(A)=\case{1}{if $x\in A$,}{0}{else.}
$$
For an orbit $\cO$ of antichains, $\chi_x(\cO)=\sum_{A\in\cO} \chi_x(A)$  is the number of times $x$ occurs in an antichain of the orbit.
We also define
$$
\chi(A) = \#A =\sum_{x\in F} \chi_x(A),
$$
so that $\chi(\cO)$ is the total number of antichain elements in an orbit.
As usual, we add a hat for the corresponding functions on ideals $I$.  For example, $\chih_x:\cI(F)\ra\{0,1\}$ is defined by
$$
\chih_x(I)=\case{1}{if $x\in I$,}{0}{else.}
$$

To state our first result in this regard we will need some additional notation for fences and tilings.  In $F$, we let
\begin{align*}
\uS_i &= \text{the set of unshared elements of segment $S_i$,}\\
s_{i,j}&= \text{the $j$th smallest element of $\uS_i$,}\\
s_i   &= \text{the unique element of $S_i\cap S_{i+1}$.}
\end{align*}
Note that we are also using $s$ for the number of segments.  But the presence or absence of a subscript will distinguish between the two uses of this notation.
Note that $s_i$ is a minimal or maximal element of $F$ depending on whether $i$ is even or odd, respectively.   For the fence $\uF(4,3,4)$ in Figure~\ref{uF(4,3,4)} we have
\begin{align*}
x_1=s_{1,1},\ x_2=s_{1,2},\ x_3=s_{1,3},\ x_4=s_1,\
x_5=s_{2,2},\\
x_6=s_{2,1},\ x_7=s_2,\
x_8=s_{3,1},\ x_9=s_{3,2},\ x_{10}=s_{3,3}.
\end{align*}

In an $\al$-tiling we let
\begin{align*}
b_i &= \text{the number of black tiles in row $i$,}\\
r_i &=  \text{the number of red tiles whose head is in row $i$.}\\
\end{align*}
So, in the first tiling of Figure~\ref{uF(4,3,4)},
$$
\barr{c|c|c}
i\   & b_i\   & r_i\ \\
\hline
1   & 4     & 1\\
2   & 5     & 1\\
3   & 4     & 0
\earr.
$$
We let $b_i=r_i=0$ if $i\not\in[s]$, where $s$ is the number of rows of the tiling.
Finally, we need the {\em Kronecker function} which, given a statement $R$, evaluates to
$$
\de(R) = \case{1}{if $R$ is true,}{0}{if $R$ is false.}
$$

\begin{lem}
\label{stats}
Let $\al=(\al_1,\al_2,\ldots,\al_s)$ with corresponding fence $F=\uF(\al)$.  Let $\cO$ be an orbit of antichains $A$ with tiling $T=\phi(\cO)$.  We denote by $I$ the ideal generated by $A$.
\begin{enumerate}
    \item[(a)]  For any $i\in[s]$ with $\al_i\ge2$ we have
    $$
    \#\cO = b_i\al_i + r_i+r_{i-1}.
    $$
    \item[(b)] Considering $\cO$ as an orbit of antichains,
    $$
    \chi_x(\cO)=
    \begin{cases}
    b_i &\text{if $x\in \uS_i$,}\\
    r_i &\text{if $x=s_i$}
    \end{cases}
    $$
    \item[(c)]  Considering $\cO$ as an orbit of ideals,
    $$
    \chih_x(\cO)=
    \begin{cases}
    b_i(\al_i-j)+r_{i-\delta(\text{$i$ even})} &\text{if $x= s_{i,j}$,}\\
    r_{2i-1}&\text{if $x=s_{2i-1}$,}\\
    \#\cO - r_{2i} &\text{if $x=s_{2i}$.}
    \end{cases}
    $$
    \item[(d)]  Considering $\cO$ as an orbit of antichains,
    $$
    \chi(\cO) = \sum_{i=1}^{s} (b_i\al_i-b_i + r_i).
    $$
    \item[(e)]  Considering $\cO$ as an orbit of ideals,    
    $$
    \chih(\cO) = \left\lfloor \frac{s-1}{2} \right\rfloor \cdot\#\cO
    + \sum_{i=1}^{s} \left[b_i \binom{\al_i}{2}
    +r_{2i-1}(\al_{2i-1}+\al_{2i}-1) -r_{2i} \right].
    $$
\end{enumerate}
\end{lem}
\begin{proof}
(a)  Since every black tile in row $i$ has $\al_i-1$ squares, the number of black squares in that row is $b_i(\al_i-1)$.  From Definition~\ref{TilDef} (a), the number of yellow tiles in row $i$ is also $b_i$ as long as $\al_i\ge2$.  Finally, there are red squares in row $i$ from both red tiles whose head is in row $i-1$ and those whose head is in row $i$.  Since the size of the orbit is the number of squares in row $i$, we have
$$
\#\cO=b_i(\al_i-1)+b_i+r_{i-1}+r_i,
$$
which simplifies to the given quantity.

\medskip

(b)  From the proof of Lemma~\ref{tiling}, we see that if $x\in\uS_i$ then a square of $T$ corresponding to $x$ appears exactly once in each black tile in row $i$.  So the first case follows. 
On the other hand, the appearances of $x=s_i$, the unique element of $S_i\cap S_{i+1}$, in $\cO$ are in bijection with red tiles covering rows $i$ and $i+1$.  But that is exactly what is counted by $r_i$.

\medskip

(c)  An element $x$ will appear in an ideal $I$ if and only if $x\lte y$ for some element $y$ of the antichain of maximial elements of $I$.  So
\begin{equation}
\label{chih_x}
    \chih_x(\cO) = \sum_{y\gte x} \chi_y(\cO).
\end{equation}

First consider the case when $x=s_{i,j}$.  Then $y\gte x$ if and only if $y=s_{i,k}$ for $k\ge j$ or $y$ is the maximal element of $S_i$.  And the latter occurs if $y=s_i$ when $i$ is odd or $y=s_{i-1}$ if $i$ is even.  Since $s_{i,k}$ corresponds to the $k$th black square of a black tile, there are $(\al_i-1)-(j-1)=\al_i-j$ elements in a given black tile satisfying the first possibility when $y\gte x$.  And for the second, we have a unique tile with its head in the appropriate row each time $y$ appears.  Combining this with equation~\eqref{chih_x} and part (b) gives the first case of the desired formula.

When $x$ is shared, it is either a maximal element and thus of the form $x=s_{2i-1}$, or a minimal element and so $x=s_{2i}$ for some $i$.  Maximal elements of $F$ are only in an ideal when they are in the corresponding antichain.  These correspond to the red tiles counted by $r_{2i-1}$.  A minimal element $x=s_{2i}$ will be in every ideal, of which there are $\#\cO$, except the ones whose antichain contains no elements in either $S_{2i}$ or $S_{2i+1}$.  In $T$ this corresponds to having yellow tiles in a given column in both rows $2i$ and $2i+1$.  But by part (b) of the definition of $\al$-tilings, these pairs of yellow tiles are in bijection with red tiles having their head in row $2i$.  This accounts for the $-r_{2i}$ term in the expression we wished to obtain.

\medskip

(d) Using the counts from (b) we obtain
\begin{align*}
\chi(\cO)
&=\sum_{x\in F} \chi_x(\cO)\\
&=\sum_{i=1}^s \left[\sum_{x\in\uS_i} b_i+\sum_{x=s_i} r_i\right]\\
&=\sum_{i=1}^s [b_i(\al_i-1)+r_i]
\end{align*}
as desired.

\medskip

(e) This equation follows from (c) in much the same way that (d) was obtained from (b).  So the details are omitted.
\end{proof}


\section{Homomesies for arbitrary and self-dual fences}
\label{has}

We now use Lemma~\ref{stats} to demonstrate various homomesy  results that hold for arbitrary fences.  We also prove a lemma about rowmotion on an arbitrary self-dual poset which will be useful in the sequel.

\begin{thm}
\label{homo}
Let $\al=(\al_1,\al_2,\ldots,\al_s)$ with corresponding fence $F=\uF(\al)$ and let $\cO$ be any rowmotion orbit of $F$.
\begin{enumerate}
    \item[(a)] If $x,y\in\uS_i$ for some $i$, then $\chi_x-\chi_y$ is $0$-mesic.
    \item[(b)]  If $x\in\uS_i$, $y=s_i$ and $z=s_{i-1}$, then $\al_i\chi_x+\chi_y+\chi_z$ is $1$-mesic. 
        \item[(c)] If $x=s_{1,j}$ and $y=s_{1,k}$, then $k\chih_x-j\chih_y$ is $(k-j)$-mesic.
    \item[(d)] Let $x=s_{2i-1}$ and $y=s_{2j}$.  If $r_{2i-1}=r_{2j}$ for all orbits $\cO$, then
    $\chih_x+\chih_y$ is $1$-mesic.
    \item[(e)]  If $s$ is odd and all the $\al_i$ are even, then $\#\cO$ is even for all orbits $\cO$.
    \item[(f)]  If $\al_i=2$ for all $i\in[s]$, then $\chi$ is $s/2$-mesic.
\end{enumerate}
\end{thm}
\begin{proof}
(a)  By Lemma~\ref{stats} (b) we have
$\chi_x(\cO)=b_i=\chi_y(\cO)$ and the result follows.

\medskip

(b)  Using Lemma~\ref{stats} (a) and (b) we obtain
$$
\al_i\chi_x(\cO)+\chi_y(\cO)+\chi_z(\cO)
=\al_i b_i + r_i + r_{i-1} =\#\cO,
$$
which is equivalent to what we wished to prove.

\medskip

(c)  From  Lemma~\ref{stats} (a) and (c) we see that
\begin{align*}
k\chih_x(\cO)-j\chih_y(\cO)
&=k[b_1(\al_1-j)+r_1] - j [b_1(\al_1-k)+r_1]\\
&=(k-j)[b_1\al_1+r_1]\\
&=(k-j)\cdot\#\cO.
\end{align*}

\medskip

(d) By Lemma~\ref{stats} (a) and (c) again, as well as the assumption on red tiles,
$$
\chih_x(\cO)+\chih_y(\cO) = r_{2i-1}+(\#\cO-r_{2j}) =\#\cO.
$$

\medskip

(e)  We will calculate the number of squares in the tiling $T=\phi(\cO)$ in two different ways.  Since all the $\al_i$ are even, we have that $\al_i-1\ge 1$ for all $i$.  This implies that there are black tiles in every row.  From Definition~\ref{TilDef} (a) we see that there is a matching between the black tiles and the yellow tiles in row $i$ for all $i$.  And a matched pair covers $\al_i$ squares, which is an even number.  By definition, red tiles cover two squares.  So the total number of squares in $T$ is even.

But since $T$ lies on an $s\times\#\cO$ cylinder, the number of squares in $T$ is also $s\cdot\#\cO$.  The fact that $s$ is odd forces $\#\cO$ to be even.

\medskip

(f)  From Lemma~\ref{stats} (a) and the assumption on $\al$ 
we have
$$
2 b_i + r_i + r_{i-1} =\#\cO
$$
for all orbits $\cO$ and all $i\in[s]$.  Summing these equations and using Lemma~\ref{stats} (b) gives
$$
s\cdot\#\cO
=2\sum_i b_i + \sum_i r_i +\sum_i r_{i-1}
=2\left(\sum_i b_i + \sum_i r_i\right)
=2\cdot\chi(\cO)
$$
from which the desired homomesy follows.
\end{proof}

\bfi
\begin{tikzpicture}
\draw (0,1)--(2,3)--(5,0)--(7,2);
\fill(0,1) circle(.1);
\draw[fill=white](1,2) circle(.1);
\draw[fill=white](2,3) circle(.1);
\fill(3,2) circle(.1);
\fill(4,1) circle(.1);
\fill(5,0) circle(.1);
\draw[fill=white](6,1) circle(.1);
\draw[fill=white](7,2) circle(.1);
\draw(2,0) node{$I$};
\draw(8,1.5) node{$\mapsto$};
\end{tikzpicture}
\begin{tikzpicture}
\draw (0,1)--(2,3)--(5,0)--(7,2);
\fill(0,1) circle(.1);
\fill(1,2) circle(.1);
\draw[fill=white](2,3) circle(.1);
\draw[fill=white](3,2) circle(.1);
\draw[fill=white](4,1) circle(.1);
\fill(5,0) circle(.1);
\fill(6,1) circle(.1);
\draw[fill=white](7,2) circle(.1);
\draw(2,0) node{$\Ib$};
\end{tikzpicture}
\capt{The ideal complement map \label{Ibar}}
\efi

We now prove a general result about the $\chih$ statistic in orbits of self-dual posets which we will use later.  To state it, we need some definitions.  Suppose our poset $P$ is self-dual. In particular, this will be true if $P=F(\al)$ where $\al$ is a palindrome (that is, equal to its reversal) with an odd number of parts.  So there is an order-reversing  bijection $\ka:P\ra P$.  Note that  $I\in\cI(P)$ if and only if $\ka(I)\in\cU(P)$.
This permits us to define the {\em ideal complement} (with respect to $\ka$) of $I\in\cI(P)$ as
$$
\Ib = \ka\circ c(I).
$$
See Figure~\ref{Ibar} for an example in $\uF(3,3,3)$ where circles are black or white depending on whether they are in the ideal or not, respectively.  The relationship with rowmotion is as follows.
\ble
\label{rho+bar}
Let $P$ be self-dual and fix an order-reversing bijection 
$\ka:P\ra P$.  Then for all $I\in\cI(P)$ we have
$$
\rhoh^{-1}(\Ib) = \ol{\rhoh(I)},
$$
where the ideal complements are with respect to $\ka$.
\ele
\begin{proof}
Since  $\ka$ and $c$ commute, they can be applied to get the ideal complement in either order.
And $\rhoh =  \De^{-1} \circ \na \circ c$,  so it suffices to show that the following diagram commutes:
$$
\begin{tikzcd}
I \arrow[r, "c"] 
& U \arrow[r, "\nabla"] \arrow[d, "\kappa"] 
& A \arrow[r, "\Delta^{-1}"] \arrow[d, "\kappa"] 
& J \arrow[d, "\kappa"] &   \\
& \overline{I} \arrow[r, "\De"]          
& B \arrow[r, "\nabla^{-1}"]                       
& V \arrow[r, "c"]   
& K.
\end{tikzcd}
$$

To see that the first square of the diagram commutes, recall that $\ka$ is an order-reversing bijection.  Thus it sends the minimal elements of any subset $S\sbe P$ bijectively to the maximal elements of $\ka(S)$.  The proof of commutativity for the second square is similar and so omitted.
\end{proof}

\begin{cor}
\label{DualOrb}
Let $P$ be self-dual with $n=\#P$, and fix an order-reversing bijection $\ka:P\ra P$.  Let $I\in\cI(P)$.
\begin{enumerate}
    \item[(a)] If  $I,\Ib\in\cO$ for some orbit $\cO$, then
    $$
    \frac{\chih(\cO)}{\#\cO}= \frac{n}{2}.
    $$
    \item[(b)] If $I\in\cO$ and $\Ib\in\cOb$ for some orbits $\cO$ and $\cOb$  with $\cO\neq\cOb$, then $\#\cO=\#\cOb$ and
    $$
    \frac{\chih(\cO)+\chih(\cOb)}{\#\cO+\#\cOb}= \frac{n}{2}.
    $$
\end{enumerate}
\end{cor}
\begin{proof}
(a)  We claim that the hypothesis on $\cO$ implies that, for any $J\in\cO$, we must have $\Jb\in\cO$.  Indeed, $J=\rhoh^j(I)$ for some integer $j$.  So, by Lemma~\ref{rho+bar},
$$
\Jb=\ol{\rhoh^j(I)}=\rhoh^{-j}(\Ib)\in\cO,
$$
since $\Ib\in\cO$.  Thus $\cO$ can be partitioned into pairs $\{I,\Ib\}$ with $I\neq \Ib$, and singletons $\{I\}$ with $I=\Ib$.  In each pair and singleton, the average of $\chih$ is $n/2$, so the same is true of the orbit.

\medskip

(b) Suppose $\cO=\{I_1,\ldots,I_k\}$.  Similar reasoning as in (a) shows that the hypothesis on $I,\Ib$ implies that
$\cOb=\{\Ib_1,\ldots,\Ib_k\}$.  Now the average of $\chih$ in each pair $\{I_i,\Ib_i\}$ is equal to $n/2$, which implies the same is true of $\cO\uplus\cOb$.
\end{proof}

To turn this last result into a homomesy, take a self-dual $P$ with a given order-reversing bijection  $\ka:P\ra P$.
Consider the group generated by the action of $\rhoh$ and the ideal complement map $I\mapsto\Ib$.  The orbits of this action will be called {\em dihedral group orbits}.  Note that from the proof of Corollary~\ref{DualOrb},
each dihedral group orbit is either an orbit of the action of $\rhoh$ or a union of two such orbits.  
The next result follows immediately from Corollary~\ref{DualOrb}.
\begin{thm}
\label{n/2-mes}
Let $P$ is self-dual with $n=\#P$ and fix an order-reversing bijection $\ka:P\ra P$.  Then $\chih$ is $(n/2)$-mesic on dihedral group orbits.\hqed
\end{thm}


\section{Fences with few segments}
\label{ffs}

In this section we will consider fences with at most $5$ segments.  For certain compositions $\al$, we will completely describe the orbit sizes and the number of orbits.  We will also calculate the values of $\chi$ and $\chih$, revealing a number of homometries.

We will need an expression for the number of ideals (equivalently, antichains) in $F=\uF(\al)$.  This is provided by the next lemma.
\begin{lem}\label{recurrenceI}
Let  $\al=(\al_1,\ldots,\al_s)$ where $s\ge2$.  Then
$$
\#\cI(\al)=\al_s\cdot\#\cI(\al_1,\ldots,\al_{s-1})
+\#\cI(\al_1,\ldots,\al_{s-2}).
$$
\end{lem}
\begin{proof}
We will assume that $s$ is odd; the proof in the case where $s$ even is similar. 
Let $x$ be the common element of the last two segments which, by the assumption on $s$, is a minimal element of $F=\uF(\al)$.  Then for any $I\in\cI(\al)$ we have $x\in I$ or $x\not\in I$.  Since $x$ is minimal, the ideals with $x\in I$ are in bijection with the ideals of $F-\{x\}\iso\uF(\al_1,\ldots,\al_{s-1})\uplus\uF(\al_s-1)$.  
Since for any two  posets $P,Q$  we have $\cI(P\uplus Q) =\cI(P)\times \cI(Q)$,
this accounts for the first term of the recursion.  On the other hand, if $x\not\in I$, then $I$ does not contain any element on either of the last two segments of $F$.  So these ideals contribute the second term, and we are done.
\end{proof}

The following corollary is easy to obtain from Lemma~\ref{recurrenceI} using straightforward computations, and so it is given without proof. 
\begin{cor}
\label{ideals}
We have the following ideal counts.
\begin{enumerate}
    \item[(a)] If $\al=(a,b)$, then
    $$
    \#\cI(\al) = ab+1.
    $$
    \item[(b)] If $\al=(a,b,c)$, then
    $$
    \#\cI(\al) = abc+a+c.
    $$
    \item[(c)] If $\al=(a,b,c,d)$, then
    $$
    \#\cI(\al) = abcd+ab+ad+cd+1.
    $$
        \item[(d)] If $\al=(a,b,c,d,e)$, then

    \eqed{
    \#\cI(\al) = abcde+abc+abe+ade+cde+a+c+e.
    }
\end{enumerate}
\end{cor}

\begin{figure}
\centering
\begin{tikzpicture}[scale=0.5]
\draw[very thin] (0,0) grid (21,-2);
\secondrow \ro \re{3} \ro \re{3} \ro \re{3} \ro \re{3} \ro \re{3} \blank
\firstrow \ro \re{4} \ro \re{4} \ro \re{4} \ro \re{4} \twobyone;
\end{tikzpicture}
\caption{The orbit $\cO$ of length $21$ in $\uF(5,4)$, with $\chi(\cO)=32$ antichain elements.}
\label{uF(5,4)}
\end{figure}
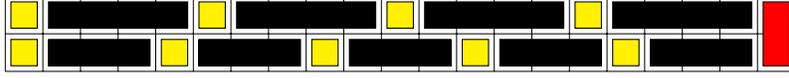

We start with the case of two segments.  Note that parts (c) and (d) of the next result are homometries which are not homomesies.  
\begin{thm}
\label{uF(a,b)}
Rowmotion on $\uF(a,b)$ has the following properties. 
\begin{enumerate}
\item[(a)] All orbits $\cO$ have size 
$\ell=\lcm(a,b)$
except for one, $\cO'$, which is the orbit of the empty set and has size $\ell+1$.
\item[(b)] The number of orbits is $\gcd(a,b)$.
\item[(c)] For any orbit of size $\ell$, 
$$
\chi(\cO)=\frac{2ab-a-b}{\gcd(a,b)}:=m.
$$
For the orbit of size $\ell+1$, 
$$
\chi(\cO') = m+1.
$$
\item[(d)]  For any orbit of size $\ell$, 
$$
\chih(\cO)=\frac{\ell(a+b-2)}{2}.
$$
For the orbit of size $\ell+1$, 
$$
\chih(\cO')=\frac{(\ell+2)(a+b-2)}{2}+1.
$$
\end{enumerate}
\end{thm}
\begin{proof}
(a)  There is only one antichain containing the unique shared element $s_1$ of $\uF(a,b)$.  So there is a unique tiling $T'$ with a red tile and all other tilings $T$ consist only of yellow and black tiles. Figure~\ref{uF(5,4)} displays $T'$ for $\uF(5,4)$.  

Consider a tiling $T$ of the second type.  From Definition~\ref{TilDef} (a), the first row can be partitioned into pairs consisting of a black tile and a yellow tile which together cover $a$ squares.  Similarly the second  row has a partition into blocks of $b$ squares.  So $T$ must have $\ell=\lcm(a,b)$ columns.  In the same manner we see that $T'$ must have one more column for the red tile.  This implies the desired orbit lengths.

\medskip

(b)  Let $k$ be the number of orbits.  Using part (a) and Corollary~\ref{ideals},
$$
ab+1=\#\cI(a,b) = (k-1)\ell+(\ell+1)=k\ell+1.
$$
It follows that $k=\gcd(a,b)$.

\medskip

(c)  From the description of the tilings in (a), for all orbits we have
$$
b_1=\frac{\lcm(a,b)}{a}=\frac{b}{\gcd(a,b)}.
$$
Similarly $b_2=a/\gcd(a,b)$.  Finally $r_1=0$ or $1$ for the orbits $\cO$ and $\cO'$, respectively.  The result now follows by substitution into Lemma~\ref{stats} (d), using that $\al_1=a$ and $\al_2=b$. 

\medskip

(d)  These equations follow from (a) and Lemma~\ref{stats} (e) in much the same way as the derivation in part (c).  We leave the details to the reader.
\end{proof}

In the next result we consider certain fences with three segments.  We do not explicitly state the values of the $\chi$ statistic  for the homometries, but these can be computed from the numbers of tiles of each color, which are given in the proof.   An example follows the demonstration.

\begin{thm}
\label{uF(a,b,a)}
Consider $\uF(a,b,a)$  and define
$$
g=\gcd(a,b),\quad
\ab = a/g,\quad
\bba = b/g,\quad
\ell=\lcm(a,b).
$$
Since $\ab,\bba$ are relatively prime, there exists a smallest positive integer $m$ such that
$$
m\ab = q \bba +1
$$
for some positive integer $q$.  Then the orbits of rowmotion on $\uF(a,b,a)$ can be partitioned by length into three sets $\cS,\cM,\cL$, which we call {\em small}, {\em medium}, and {\em large}, having the following properties.
\begin{enumerate}
    \item[(a)] We have
    $$
    \#\cS = \ab(g-1)^2,\quad
    \#\cM = q,\quad
    \#\cL = \ab-q,
    $$
    and
    $$
    \#\cO =
    \begin{cases}
    \ell&\text{if $\cO\in\cS$},\\
    a(2b-2\bba+m)+g&\text{if $\cO\in\cM$},\\
    a(2b-\bba+m)+g&\text{if $\cO\in\cL$}.
    \end{cases}
    $$
    \item[(b)] For rowmotion on antichains, $\chi$ is homometric. 
    \item[(c)] For rowmotion on ideals, $\chih$ is $n/2$-mesic where $n=\#\uF(a,b,a)=2a+b-1$.
\end{enumerate}
\end{thm}
\begin{proof}
(a) We will construct the corresponding tilings, where we will number the columns from left to right starting with column $0$.  We can use any integer to refer to a column by using its standard representative modulo $\#\cO$, the number of columns.
Figure~\ref{uF(4,3,4)} gives an example of the tilings for $\uF(4,3,4)$, where $\cS=\emp$ because $g=1$.

We first construct the tilings $T$ for the orbits in $\cS$. 
We also standardize these tilings so that the first row begins with a yellow tile.
The first row of $T$ will consist of $\ell/a$ copies of a pair consisting of a yellow tile followed by a black tile of length $a-1$.  So $T$ has length $\ell$.
We now tile the second row with $\ell/b$ copies of a paired yellow tile and black tile of length $b-1$ positioned in such a way that no two yellow tiles from the first two rows are in the same column.  Finally tile the third row with alternating yellow and length $a-1$ black tiles, requiring that no two yellow tiles from the last two rows are in the same column.  Clearly this is a valid tiling of length $\ell$.  So, for this case, it remains to calculate the number of rotationally different tilings obtained.

The first square of a black tile in the second row could be placed under any one of the $a$  squares of the union of a yellow tile and an adjacent black tile in the first row.  Since there are $\ell/b$ black tiles in the second row, any placement of the first black tile will be rotationally equivalent to that many other placements.  So the number of nonequivalent placements is $a/(\ell/b)=g$.  But one of these will be the rotation class which has two yellow tiles in the same column. Thus the count for the placement of the second row is really $g-1$.  Similarly one sees that the number of inequivalent placements for the third row is $\ab(g-1)$, giving $\ab(g-1)^2$ as the final total.

We will now construct $\ab$ tilings  $T_1,\ldots,T_{\ab}$,
where the $T_k$ for $k\le (\ab m-1)/b$ will have $a(2b-2\bba+m)+g$ columns and correspond to orbits in $\cM$, and those for the remaining values of $k$ will have $a(2b-\bba+m)+g$ columns and correspond to orbits in $\cL$.  All these tilings will have $g$ red tiles with their heads in row $1$ and $g$ red tiles with their heads in row $2$.  For all $k$, the former tiles of $T_k$ will be in columns
$$
0,\ \ell+1,\ 2(\ell+1),\ \ldots,\ (g-1)(\ell+1),
$$
and the latter will be in columns
$$
-kb,\ -kb-(\ell+1),\ -kb-2(\ell+1),\ \ldots,\
-kb-(g-1)(\ell+1),
$$
where the indices are taken modulo the number of columns  of the tiling.  It is now a matter of simple arithmetic to check that these placements of red tiles can be (uniquely) extended to full tilings.  Also, these tilings are distinct, since the red tiles with heads in the first row of all $T_k$ are fixed in this description, while the red tiles with heads in the second row vary.

Finally we must verify that these tilings correspond to all orbits by adding up the orbit sizes and making sure they sum to $a(ab+2)$, which is the total number of antichains in $\uF(a,b,a)$ by Corollary~\ref{ideals} (b).  This is a simple, if tedious, computation and is left to the reader.

\medskip

(b)  We see from their descriptions in part (a) that these tilings have the following parameters.
$$
\begin{array}{c|ccc}
\text{type of orbit}
        & b_1=b_3   & b_2       & r_1=r_2\\
\hline
\cS     &\bba        &\ab    &0\rule{0pt}{12pt}\\
\cM     &2b-2\bba+m  &2a-2\ab+q  &g\\
\cL     &2b-\bba+m   &2a-\ab+q   &g
\end{array}
$$
Since $\chi$ only depends on these parameters by Lemma~\ref{stats} (d), and they are constant on orbit sizes, this statistic is homometric.

\medskip

(c)  For the orbits  $\cO\in\cS$, using Lemma~\ref{stats} (e) and the previous table, we obtain
$$
\chih(\cO)
=\ell+2 \bba\binom{a}{2}+\ab\binom{b}{2}
=\ell+\ell(a-1)+\frac{\ell(b-1)}{2}
=\ell\cdot \frac{2a+b-1}{2}
=\ell\cdot \frac{n}{2}
$$
as desired.

For the orbits in $\cO\in\cM\cup\cL$ we will use Corollary~\ref{DualOrb} (a).  It suffices to find an ideal $I\in\cO$ with $\bar{I}\in\cO$.  Let $I$ be the ideal corresponding to the zeroth column of the corresponding tiling $T$.  We claim that the ideal $J$ corresponding to column $-kb-1$ satisfies $J=\bar{I}$.  By the choice of $I$, it must contain $s_1$, the intersection of the first two segments.
Suppose $I$'s maximum element on the third segment is the $t$th unshared element from the bottom, where $t=0$ if $I$ contains no such element.

The column for $J$ comes just before the red tile in column $-kb$.  So $J$ can have no elements in either the second or third segments.  Furthermore, in the first row there is an alternating sequence of black and yellow tiles, starting with a black tile whose last square is in the last column and working backwards, because of the red tile in column $0$.  Similarly, the red tile in column $-kb$ guarantees a similar sequence of the same length starting in column $-kb+1$ and going forwards in the third row.  It follows that $J$ has the $(a-t-1)$st unshared element on segment one as its maximal element.  But this matches the description of $\bar{I}$ given by applying $c$ and then applying $\kappa$.
\end{proof}

\begin{figure}
\centering
\begin{tikzpicture}[scale=0.5]
\draw[very thin] (0,0) grid (6,-3);
\firstrow \re{1} \ro \re{1} \ro \re{1} \ro
\secondrow \ro \re{5}
\thirdrow \re{1} \ro \re{1} \ro \re{1} \ro

\begin{scope}[shift={(8,0)}]
\draw[very thin] (0,0) grid (22,-3);
\firstrow \ro\re{1}\ro\re{1}\ro\re{1}\ro\re{1}\ro\re{1}\ro\re{1}\ro\re{1}\twobyone\ro\re{1}\ro\re{1}\ro\re{1}\twobyone
\secondrow \ro\twobyone\re{5}\ro\twobyone\re{5}\blank\ro\re{5}\blank
\thirdrow \ro\blank\re{1}\ro\re{1}\ro\re{1}\ro\blank\re{1}\ro\re{1}\ro\re{1}\ro\re{1}\ro\re{1}\ro\re{1}\ro\re{1}
\end{scope}

\end{tikzpicture}
\capt{The two orbits of $\uF(2,6,2)$.}
\label{uF(2,6,2)}
\end{figure}

To illustrate the previous result, consider the fence $\uF(2,6,2)$.
So 
$$
g=\gcd(2,6)=2,\quad
\ab=2/2 = 1,\quad
\ol{b}= 6/2 = 3,\quad
\ell=\lcm(2,6)=6.
$$
So $m\ab=q\ol{b}+1$ becomes $m=3q+1$, and the smallest $m$ satisfying this relation for some positive $q$ is when $q=1$ and $m=4$.  It follows that
$$
\#\cS=1\cdot(2-1)^2 =1,\quad
\#\cM = 1,\quad
\#\cL = 1-1 = 0.
$$
The single small and medium orbits are shown in Figure~\ref{uF(2,6,2)}.

We next consider the case of a $4$-segment fence with all parts of $\al$ equal.  An example of the orbits is given  in Figure~\ref{uF(3,3,3,3)}.

\begin{thm}
\label{uF(a,a,a,a)}
The orbits of rowmotion on $\uF(a,a,a,a)$ can be partitioned into four sets $\cS,\cM,\cL,\cG$ called
{\em small}, {\em medium}, {\em large}, and {\em gigantic} having the following properties.
\begin{enumerate}
    \item[(a)] We have
    $$
    \#\cS = (a-1)^3,\quad
    \#\cM = a,\quad
    \#\cL = 1,\quad
    \#\cG = a-1,
    $$
    and
    $$
    \#\cO =
    \begin{cases}
    a&\text{if $\cO\in\cS$},\\
    a+1&\text{if $\cO\in\cM$},\\
    a^2+a+1&\text{if $\cO\in\cL$},\\
    3a^2+a&\text{if $\cO\in\cG$}.
    \end{cases}
    $$
    \item[(b)]  The statistics $\chi$ and $\chih$ are homometric under rowmotion, with values given in the following chart:
    $$
    \barr{c|cccc}
            &\cS    &\cM    &\cL    &\cG\\
    \hline
    \chi(\cO)& 4a-4 & 4a-2  & 4a^2-a& 12a^2-11a+2\\
    \chih(\cO)&2a^2-a&2a^2+3a-1&2a^3+3a-1&6a^3-a
    \earr
    $$
\end{enumerate}
\end{thm}
\begin{proof}
(a) We will construct the corresponding tilings, labeling  the columns starting at 0. For $1\le i\le 4$, let $j_i$ denote the smallest nonnegative number so that row $i$ of the tiling has a yellow square in column $j_i$. 
We set $j_1=0$, since one can always choose the first antichain in the orbit to have a yellow square in the top row. Let us consider four types of orbits as follows.

\medskip

$\cS$: $0\neq j_2\neq j_3\neq j_4$. In such an orbit, each row simply consists of one yellow and one black tile, and the orbit has length $a$. For $2\le i\le 4$, there are $a-1$ possibilities for $j_i$, since it can be any integer between $1$ and $a$ other than $j_{i-1}$. Thus, there are $(a-1)^3$ such orbits. 
\medskip

$\cM$: $0= j_2\neq j_3= j_4$. In such an orbit, rows 1 and 2 both consist of a yellow tile in column 0, followed by a black tile, followed by a red tile with its head in row 1. Rows 3 and 4 have the same configuration, shifted $j_3$ positions to the right. Therefore, such an orbit has length $a+1$. Since $j_3$ can be any integer between $1$ and $a$, there are $a$ such orbits. 
\medskip

$\cL$: $0= j_2= j_3= j_4$. We will construct a unique orbit of this type where every row intersects at least one red tile.  So it suffices to specify the positions of these tiles. There are two red tiles in column $a(a+1)$ (the rightmost column), with heads in rows 1 and 3. Additionally, there are $a$ red tiles with their head in row 2, in columns $1+r(a+1)$ for $0\le r\le a-1$. It is straightforward to check that these red tiles can be extended to a full tiling.
 
\medskip

$\cG$: $0= j_2= j_3\neq j_4$. Let us write $j$ instead of $j_4$ for short. 
For each value of $j$ satisfying $1\le j\le a-1$, we construct an orbit of length $3a^2+a$.
Again, it suffices to specify the positions of the red tiles, since every row intersects at least one red tile. From left to right, the red tiles are in the following positions:
\begin{itemize}
    \item In columns $1+r(a+1)$ for $0\le r\le j-1$, with head in row 2.
    \item In columns $r(a+1)$ for $j\le r\le a+j-1$, with head in row 3.
    \item In columns $2+r(a+1)$ for $a+j-1\le r\le 2a-2$, with head in row 2.
    \item In columns $1+r(a+1)$ for $2a-1\le r\le 3a-2$, with head in row 1.
\end{itemize}

It can be checked that this configuration of red tiles can be extended to a valid tiling. Note  that the $a-1$ orbits obtained for the different values of $j$ are all distinct. This is because the positions of the red tiles with heads in row $1$ are independent of $j$, whereas the sets of positions of the other red tiles are different for each $j$.

Next we show that we have obtained tilings for all the orbits. The total number of antichains in the orbits that we have constructed is
$$
(a-1)^3\cdot a+ a\cdot (a+1)+(a^2+a+1)+(a-1)\cdot(3a^2+a)=a^4+3a^2+1,
$$
which equals $\#\cI(a,a,a,a)$ by Corollary~\ref{ideals} (c).

\medskip

(b) From the constructions in (a) we see that, for given $i$, the parameters $b_i$ and $r_i$ are constant in each of the four classes of orbits.  Homometry follows immediately.  To compute the statistics, we note the following specific values.
$$
\begin{array}{c|cccc}
\text{type of orbit}
        & b_1=b_4   & b_2=b_3   & r_1=r_3   & r_2\\
\hline
\cS     & 1         & 1         & 0         & 0\rule{0pt}{12pt}\\
\cM     & 1         & 1         & 1         & 0\\
\cL     & a+1       & a         & 1         & a\\
\cG     & 3a        & 3a-1      & a         &a
\end{array}
$$
Now the formulas in Lemma~\ref{stats} (d) and (e) complete the proof.
\end{proof}

\begin{figure}
\centering
\begin{tikzpicture}[scale=0.5]
\draw[very thin] (0,0) grid (3,-4);
\firstrow \ro \re{2}
\secondrow \releft{1} \ro \reright{1}
\thirdrow \ro \re{2}
\fourthrow \re{2} \ro

\begin{scope}[shift={(6,0)}]
\draw[very thin] (0,0) grid (4,-4);
\firstrow \ro \re{2} \twobyone
\secondrow \ro \re{2} \blank
\thirdrow \releft{1} \ro \twobyone \reright{1}
\fourthrow \releft{1} \ro \blank \reright{1}
\end{scope}

\begin{scope}[shift={(13,0)}]
\draw[very thin] (0,0) grid (13,-4);
\firstrow \ro \re{2} \ro \re{2}\ro \re{2}\ro \re{2} \twobyone
\secondrow \ro \twobyone \re{2} \ro \twobyone \re{2} \ro \twobyone \re{2} \blank
\thirdrow \ro \blank \re{2} \ro \blank \re{2} \ro \blank \re{2} \twobyone
\fourthrow \ro \re{2} \ro \re{2}\ro \re{2}\ro \re{2} \blank
\end{scope}

\begin{scope}[shift={(-2,-5)}]
\draw[very thin] (0,0) grid (30,-4);
\firstrow \ro\re{2} \ro\re{2} \ro\re{2} \ro\re{2} \ro\re{2} \ro\re{2} \ro\re{2}\twobyone \ro\re{2}\twobyone \ro\re{2}\twobyone
\secondrow \ro\twobyone\re{2} \ro\twobyone\re{2} \ro\re{2} \ro\re{2} \ro\re{2} \ro\twobyone\re{2}\blank \ro\re{2}\blank \ro\re{2}\blank
\thirdrow \ro\blank\re{2} \ro\blank\re{2} \twobyone\ro\re{2} \twobyone\ro\re{2} \twobyone\ro\blank\re{2} \ro\re{2} \ro\re{2} \ro\re{2}
\fourthrow \re{2} \ro\re{2} \ro\re{2}\blank \ro\re{2}\blank \ro\re{2}\blank \ro\re{2} \ro\re{2} \ro\re{2} \ro\re{2} \ro
\end{scope}
\end{tikzpicture}
\capt{An orbit of $\uF(3,3,3,3)$ in each of the four sizes.}
\label{uF(3,3,3,3)}
\end{figure}

We end this section by considering certain fences with $5$ segments.  An example for $\uF(3,1,3,1,3)$ is given in Figure~\ref{uF(3,1,3,1,3)}.

\begin{thm}
\label{uF(a,1,a,1,a)}
The orbits of rowmotion on $\uF(a,1,a,1,a)$  can be partitioned into 
three sets $\cS,\cM,\cL$  called {\em small}, {\em medium}, and {\em large} having the following properties.
\begin{enumerate}
    \item[(a)] We have
    $$
    \#\cS = 2a-2,\quad
    \#\cM = 1,\quad
    \#\cL = a,
    $$
    and
    $$
    \#\cO =
    \begin{cases}
    a+1&\text{if $\cO\in\cS$},\\
    3a+2&\text{if $\cO\in\cM$},\\
    a^2+2a&\text{if $\cO\in\cL$}.
    \end{cases}
    $$
    \item[(b)] For rowmotion on antichains, $\chi$ is homometric  with
    $$
    \chi(\cO) =
    \begin{cases}
    3a&\text{if $\cO\in\cS$},\\
    9a-3&\text{if $\cO\in\cM$},\\
    3a^2+3a-2&\text{if $\cO\in\cL$}.
    \end{cases}
    $$   
    \item[(c)] The dihedral group orbits are pairs of orbits in $\cS$ and single orbits in $\cM$ and $\cL$.
\end{enumerate}
\end{thm}
\begin{proof}
(a)
We label the columns of a tiling as usual.  Since there will be no black tiles in rows 2 and 4, every row will intersect at least one red tile.  Thus it suffices to specify the positions of these tiles.  We will omit the straightforward check that the given positions of the red tiles can be extended to a full tiling with no column repeated, and that the orbits are distinct and have a total of  $a^3+4a^2+3a$ columns, which is the total number of antichains of the fence by Corollary~\ref{ideals} (d).  We label the orbits in $\cL$ by $\cL_k$ where $1\le k\le a$.  We label the orbits in $\cS$ which have no red tile with a head in row $2$ as $\cS_k$ where $1\le k\le a-1$.
Similarly, the orbits in $\cS$ with no red tile with a head in row $3$ are labeled as $\cS_k'$  where $1\le k\le a-1$.
Every tiling in $\cS$ satisfies exactly one of these two conditions, so these two subsets have disjoint union $\cS$. Finally, the sole orbit in $\cM$ will be labeled $\cM_1$.

We now give the column numbers for the red tiles in each orbit with their head in each row.  The orbits $\cO_k$ have $r_i=a$ for $1\le i\le 4$, and so the positions will be parameterized by $q$ where $0\le q \le a-1$.  The symbol NA (not applicable) in the following table indicates that there are no red tiles for that orbit with their head in the corresponding row.
$$
\begin{array}{l||c|c|c|c}
\text{orbit}  &
\cL_k & \cM_1 & \cS_k & \cS_k'\\
\hline
\text{row $1$}&
q(a+1)          & 0, a+1    & 0     & 0\\
\text{row $2$}& 
q(a+2)+2        & -a        & \text{NA}& k+1\\
\text{row $3$}& 
q(a+2)          & a         & -1    & \text{NA}\\
\text{row $4$}& 
(q-k+1)(a+1)-k  & 0, -a-1   & k     & k\\
\end{array}
$$

\begin{figure}
\centering
\begin{tikzpicture}[scale=0.5]
\draw[very thin] (0,0) grid (4,-5);
\firstrow \twobyone \ro \re{2}
\secondrow \blank \ro \ro \ro
\thirdrow \ro \re{2} \twobyone
\fourthrow \ro \twobyone \ro \blank
\fifthrow \ro \blank \re{2}

\begin{scope}[shift={(6,0)}]
\draw[very thin] (0,0) grid (11,-5);
\firstrow \twobyone \ro \re{2} \twobyone \ro \re{2} \ro \re{2}
\secondrow \blank \ro\ro\ro \blank \ro\ro\ro \twobyone \ro\ro
\thirdrow \ro\re{2} \twobyone \ro\re{2}\ro  \blank \re{2}
\fourthrow \twobyone \ro\ro \blank \ro\ro\ro \twobyone \ro\ro\ro
\fifthrow \blank \re{2}\ro\re{2}\ro \blank \re{2}\ro
\end{scope}

\begin{scope}[shift={(1,-6)}]
\draw[very thin] (0,0) grid (15,-5);
\firstrow \twobyone \ro\re{2} \twobyone \ro\re{2} \twobyone \ro\re{2}\ro\re{2}
\secondrow \blank \ro \twobyone \ro \blank \ro\ro \twobyone \blank \ro\ro\ro \twobyone \ro\ro
\thirdrow \twobyone \ro \blank \re{2} \twobyone \ro \blank \re{2} \twobyone \ro \blank \re{2}
\fourthrow \blank \ro\ro\ro \twobyone \blank \ro \ro \twobyone \ro \blank \ro \twobyone \ro\ro
\fifthrow \ro\re{2}\ro \blank \re{2}\ro \blank \re{2}\ro \blank \re{2}
\end{scope}

\end{tikzpicture}
\capt{An orbit of $\uF(3,1,3,1,3)$ in each of the three sizes.}
\label{uF(3,1,3,1,3)}
\end{figure}

(b)  The following chart of the $r_i$ and $b_i$ values reduces this proof to a routine computation.
$$
\begin{array}{l||c|c|c|c|c|c|c|c}
\text{orbit}  &
\multicolumn{2}{c|}{\cL_k} & 
\multicolumn{2}{c|}{\cM_1} & 
\multicolumn{2}{c|}{\cS_k} & 
\multicolumn{2}{c}{\cS_k'}\\
& b_i & r_i & b_i & r_i & b_i & r_i & b_i & r_i\\ 
\hline
\text{row $1$}&
  a+1   & a &  3  &  2  &  1  &  1  &  1  &  1\\
\text{row $2$}& 
  0   & a   &  0  &  1  &  0  &  0  &  0  &  1\\
\text{row $3$}& 
  a   & a   &  3  &  1  &  1  &  1  &  1  &  0\\
\text{row $4$}& 
  0   & a   &  0  &  2  &  0  &  1  &  0  &  1  \\
\text{row $5$}&
  a+1   & 0 &  3  &  0  &  1  &  0  &  1  &  0  
\end{array}
$$

(c)  First of all, the lone orbit of size $3a+2$ must also be a dihedral group orbit  since it has no other orbit with which it could be paired.  To show that the orbits of size $a^2+2a$ are also dihedral group orbits, it suffices to find an ideal  $I\in \cL_k$ with $\bar{I}=\ka(I)\in\cL_k$.  We take as $I$ the ideal corresponding to column $0$, which consists of red tiles with their heads in rows $1$ and $3$, and the cell in row $5$ covered by the $k$th cell of a black tile for $1\le k\le a-1$ or by a yellow tile for $k=a$.  One can now verify that $\bar{I}$ corresponds to column $(a-k-1)(a+2)+1$ of $\cL_k$.

Finally, we show that $\ka$ pairs $\cS_k$
with $\cS_k'$.  Let $I$ correspond to column $a$ in $\cS_k$.  This column consists of the last cell of a black tile in row $1$, a yellow tile in row $2$, a red tile in rows $3$ and $4$, and the $(a-1)$st cell of a black tile in row $5$. It is not hard to see that $\bar{I}$ corresponds to column in $\cS_k'$ which contains a red tile with its head in row 4.
\end{proof}


\section{Conjectures and an open problem}
\label{cop}

We now consider some conjectures and an open problem in the hopes that they will spur future research.  To state the first one, call a sequence of real numbers $a_1,a_2,\ldots,a_n$ {\em palindromic} if 
$$
a_i=a_{n+1-i}
$$
for all $1\le i\le n$.  Associated with any fence $\uF(\al_1,\al_2,\ldots,\al_s)$ we have the {\em black tile sequence} $b_1,b_2,\ldots,b_s$ and the {\em red tile sequence} $r_1,r_2,\ldots,r_{s-1}$.
\begin{question}
Let $\al=(\al_1,\al_2,\ldots,\al_s)$ be palindromic and $F=\uF(\al)$.  Find necessary and/or sufficient conditions on $\al$ for the black or the red tile sequences to be palindromic for all rowmotion orbits.
\end{question}

From the proofs of Theorems~\ref{uF(a,b)}, \ref{uF(a,b,a)}, and~\ref{uF(a,a,a,a)}, we see that the black and red tile sequences are palindromic for palindromic $\al$ with  $s\le 3$, as well as for $\uF(a^4)$ where $a^s$ denotes $a$ repeated $s$ times.
We have checked by computer that $\uF(a^s)$ has palindromic tile sequences for all $a+s\le 12$ {\em except} $a=4$ and $s=8$, where the orbit containing the antichain $\{s_1,s_{2,1}\}$ (equivalently, $\{x_1,x_7\}$ indexing the entries from left to right), has black tile sequence
$$
21, 20, 18, 18, 19, 18, 19, 21
$$
and red tile sequence
$$
5, 4, 13, 4, 9, 8, 5.
$$

The black and red tile sequences are related for palindromic $\al$, as shown in the next result.  
\begin{prop}
Let $\al=(\al_1,\al_2,\ldots,\al_s)$ be palindromic where $\al_i\ge2$ for all $i\in[s]$, and let $F=\uF(\al)$.  Then for any orbit $\cO$ of $F$,
the black tile sequence  $b_1,b_2,\ldots,b_s$  is palindromic if and only if the red tile sequence
$r_1,r_2,\ldots,r_{s-1}$ is palindromic.
\end{prop}
\begin{proof}
 We will do induction on $i$.  Because all $\al_i\ge2$, Lemma~\ref{stats} (a) implies that  $\#\cO=b_i\al_i+r_i+r_{i-1}$ for  any $i\in[s]$.  Using the fact that $r_0=r_s=0$ we obtain
$$
b_1\al_1+r_1 = \#\cO = b_s\al_s + r_{s-1}.
$$
Since $\al_1=\al_s$, This can be rewritten as
$$
\al_1(b_1-b_s)=r_{s-1}-r_1.
$$
It follows that $b_1=b_s$ if and only if $r_1=r_{s-1}$ which is the base case for our induction.

The induction step is similar.  We have
$$
b_i\al_i+r_i+r_{i-1}=\#\cO=b_{s-i+1}\al_{s-i+1}+r_{s-i+1}+r_{s-i}.
$$
By induction, we can assume $r_{i-1}=r_{s-i+1}$.  And $\al_i=\al_{s-i+1}$, so
$$
\al_i(b_i-b_{s-i+1})=r_{s-i}-r_i.
$$
Thus $b_i=b_{s-i+1}$ is equivalent to $r_{s-i}=r_i$, which finishes the proof.
\end{proof}

The condition that  $\al_i\ge2$ for all $i$ in this proposition is necessary.  For example, the proof of Theorem~\ref{uF(a,1,a,1,a)} (b) shows that the black tile sequence of an orbit is always palindromic, but that the red tile sequence is not necessarily so.
Having palindromic sequences yields certain homomesies for free. 
\begin{prop}
Let $\al=(\al_1,\al_2,\ldots,\al_s)$  where $\al_i\ge2$ for all $i\in[s]$. Also let $F=\uF(\al)$ and $n=\#F$. 
If $\al$ as well as the black and red tile sequences are all palindromic, then one has the following homomesies.
\begin{enumerate}
    \item[(a)] For all $k\in[n]$ the statistic $\chi_k-\chi_{n-k+1}$ is $0$-mesic.
    \item[(b)] If $s$ is odd, then for all $k\in[n]$ the statistic $\chih_k+\chih_{n-k+1}$ is $1$-mesic.
\end{enumerate} 
\end{prop}
\begin{proof}
(a) 
First suppose $x_k\in\uS_i$ for some $i$.  Since $\al$ is palindromic, it follows that $x_{n-k+1}\in\uS_{s-i+1}$.  Now using Lemma~\ref{stats} and the fact that the black tile sequence is palindromic, we see   that
$$
\chi_k(\cO) =b_i = b_{s-i+1} = \chi_{n-k+1}(\cO)
$$
for all orbits $\cO$.
This proves the homomesy for our choice of $x_k$. The proof when $x_k=s_i$ for some $i$ is similar using  the palindromicity of the red tile sequence.  The details are omitted.

\medskip

(b)  If $x_k=s_i$ for some $i$, then this follows from our assumptions and Theorem~\ref{homo} (d).  So assume that $x_k=s_{i,j}$ for some $i,j$.  There are two similar cases depending on whether $i$ is even or odd, so we will only give details for the former.  From Lemma~\ref{stats} (c), for any orbit $\cO$ we have
$$
\chih_k(\cO) = b_i(\al_i-j) +r_{i-1}.
$$
Since $x_k$ is an unshared element on $S_i$, we have that $x_{n-k+1}$ is an unshared element on $S_{s-i+1}$.
And since $x_k$ is the $j$th unshared element from the bottom, we have that $x_{n-k+1}$ is the $j$th unshared element from the top.  So counting from the bottom of the unshared elements in  $S_{s-i+1}$  we have that $x_{n-k+1}$ is element number
$\al_{s-i+1}-j$.
Using Lemma~\ref{stats} (c) again, as well as the assumption that the $b_i$ and $r_i$ sequences are palindromic, gives
$$
\chih_{n-k+1}(\cO)=
b_{s-i+1}[\al_{s-i+1}-(\al_{s-i+1}-j)]+r_{(s-i+1)-1}
=jb_i+r_i.
$$
Finally
$$
\chih_k(\cO)+\chih_{n-k+1}(\cO)
= b_i(\al_i-j) +r_{i-1} +jb_i+r_i = \#\cO
$$
by Lemma~\ref{stats} (a).  This completes the proof of (b).
\end{proof}

We conjecture that even more homomesies are true for constant $\al$.
\begin{conj}
\label{a^s}
Let $\al=(a^s)$ and let $F=\uF(\al)$. 
\begin{enumerate}
    \item[(a)] The statistic $\chi$ is homometric.
    \item[(b)] If $s$ is odd then the statistic $\chih$ is $n/2$-mesic where $n=\#F$.
\end{enumerate}
\end{conj}
Again, we have evidence for this conjecture from the results in Section~\ref{ffs}, which prove that it holds for $s\le4$.  Furthermore, computer calculations have verified the truth of these statements for $a+s\le12$.
One might be tempted to try and prove part (b) by showing that, under the given hypotheses, the dihedral group orbits and orbits are the same, and then using Theorem~\ref{n/2-mes}. This would mean that for every ideal $I$ we would have $\Ib$ in the same orbit with $I$.  But this is not always true.  For example, if $\al=(2^7)$ and $F=\uF(\al)$, then the ideal $I=\{x_2,x_4,x_7\}$ has 
$\Ib=\{x_1,x_3,x_5\}$, which is in a different orbit.

We also note that if $a=2$ then Conjecture~\ref{a^s} (a) is implied by the following stronger statement.  This result also follows from Corollary 3.11 in a paper of Chan, Haddadan, Hopkins, and Moci~\cite{CHHM:ejo}.
\begin{prop}
\label{a=2}
In $F=\uF(2^s)$ the statistc $\chi$ is $s/2$-mesic.
\end{prop}
\begin{proof}
Let $x_i$ be the unique unshared element on $S_i$ for $i\in[s]$.
Let 
$$
\chi_i = 2\chi_{x_i} + \chi_y+\chi_z
$$
where $y,z$ are as in Theorem~\ref{homo}.  By that result, for any orbit $\cO$ we have
$\sum_i \chi_i(\cO)=s\cdot \#\cO$.  On the other hand, in the previous sum $\chi_x$ appears twice for every $x\in F$.  It follows that
$\sum_i \chi_i(\cO)= 2\chi(\cO)$ which completes the proof.
\end{proof}

Finally we mention that Defant and Lin have conjectured that the down-degree statistic is homometric in the $\nu$-Tamari lattice corresponding to a lattice path $\nu$.  See~\cite{DL:rmt} for more information.

{\em Acknowledgement.}  We would like to thank BIRS (the Banff International Research Station) for hosting the conference during which the research for this paper began.
We also thank Sam Hopkins for pointing out Proposition~\ref{a=2} and both its proofs.



\nocite{*}
\bibliographystyle{alpha}

\end{document}